\documentclass{article}
\usepackage{amssymb,amsmath,bm}
\usepackage{amsthm}
\usepackage{hyperref}
\usepackage{mathrsfs}
\usepackage{textcomp}
\usepackage{enumerate}
\usepackage{graphicx}
\usepackage{caption}

\newtheorem{theorem}{Theorem}[section]
\newtheorem{lemma}[theorem]{Lemma}
\newtheorem{proposition}[theorem]{Proposition}

\theoremstyle{remark}

\theoremstyle{remark}

\numberwithin{equation}{section}

\begin{document}

\title{\textbf{Angle structures and hyperbolic $3$-manifolds with totally geodesic boundary}
\bigskip}

\author{\medskip Faze Zhang, Ruifeng Qiu \& Tian Yang}

\date{}

\maketitle
\begin{abstract}

This notes explores angle structures on ideally triangulated compact $3$-manifolds with high genus boundary. We show that the existence of angle structures implies the existence of a hyperbolic metric with totally geodesic boundary, and conversely each hyperbolic $3$-manifold with totally geodesic boundary has an ideal triangulation that admits angle structures.
\end{abstract}


\section{Introduction}

By Perelman's proof of Thurston's Geometrization Conjecture~\cite{perelman1, perelman2, perelman3}, it is known that the interior of every compact $3$-manifold has a canonical decomposition into geometric pieces, all but a few classified pieces have a unique hyperbolic structure. To calculate the hyperbolic structure on the pieces with torus boundary, Thurston\,\cite{Thurston} introduced a system of algebraic equations, called the hyperbolic gluing equations. The idea is to ideally triangulate the manifold into Euclidean tetrahedra and realize each tetrahedron as a hyperbolic ideal tetrahedron by assigning the shape parameter, a complex number with positive imaginary part. He showed that the shape parameters satisfy the hyperbolic gluing equations; and conversely, a solution to the hyperbolic gluing equations, if exists, gives rise to the hyperbolic structure. In general, the hyperbolic gluing equations are very difficult to solve. Instead of directly solving the hyperbolic gluing equations, Casson\,\cite{Casson} introduced the notion of angle structures which assign dihedral angles to the tetrahedra that satisfy a system of liner equations and strict linear inequalities. He then showed that among all the angle structures, the one that has the maximum volume, if exists, gives rise to a solution to the hyperbolic gluing equations, and hence the hyperbolic structure. See also Lackenby\,\cite{Lackenby1}, Rivin\,\cite{Riv} and Futer-Gu\'eritaud\,\cite{FG}.

Using Thurston's Hyperbolization Theorem\,\cite{Thurston2}, Casson also showed that if an ideally triangulated $3$-manifold with torus boundary admits an angle structure, then it has a hyperbolic structure. See also\,\cite{Lackenby1}. Conversely, a recent work of Hodgson-Rubinstein-Segerman\,\cite{HRS} showed that each cusped hyperbolic $3$-manifold that satisfies some topological condition has an angled triangulation, i.e., an ideal triangulation that admits angle structures. Their construction made an essential use of Epstein-Penner's decomposition of cusped hyperbolic $3$-manifolds\,\cite{EP} and a duality theorem of Luo-Tillmann\,\cite{LT} relating the existence of angle structures and the non-existence of normal surfaces of certain type.

To calculate the hyperbolic structure on the geometric pieces with high genus boundary, Luo\,\cite{Luo1} introduced the corresponding notion of angle structures, which he termed the linear hyperbolic structures. According to \cite{Luo1}, an \emph{angle structure} on an ideally triangulated $3$-manifold $M$ with boundary consisting of surfaces of negative Euler characteristic is an assignment of a positive real number to each edge of each tetrahedron of the triangulation, called the \emph{dihedral angle}, so that the sum of the dihedral angles around each edge is $2\pi,$ and the sum of the dihedral angles at the three edges in a tetrahedron adjacent to a vertex is strictly less than $\pi.$ By the work of Bao-Bonahon\,\cite{BB}, the second condition characterizers the dihedral angles of (truncated) hyperideal tetrahedra  in the $3$-dimensional hyperbolic space $\mathbb H^3.$ Therefore, an angle structure makes each individual tetrahedron in the ideal triangulation a hyperideal tetrahedron in $\mathbb{H}^3.$ The volume of an angle structure is  defined to be the sum of the hyperbolic volume of the hyperideal tetrahedra determined by the assigned dihedral angles. Luo\,\cite{Luo1} showed that the maximum volume angle structure, if exists, gives rise to the hyperbolic metric on $M$ with totally geodesic boundary.

Our main result in this paper proves the counterpart of Casson and Hodgson-Rubinstein-Segerman's results in the setting of angle structures  on $3$-manifolds with high genus boundary. We have

 \begin{theorem} \label{main} Let $M$ be a compact $3$-manifold with boundary consisting of surfaces of negative Euler characteristic. Then the following are equivalent:
 \begin{enumerate}[(1)]
 \item $M$ admits an ideal triangulation that supports angle structures,
\item $M$ admits a hyperbolic metric with totally geodesic boundary.
\end{enumerate}
\end{theorem}

We are informed by Lackenby that (1) implying (2) was first known in \cite{Lackenby3}. In comparison with \cite{HRS}, we do not require any topological condition on the manifold. The constructions are inspired by the works of Lackenby\,\cite{Lackenby1,Lackenby2} and Hodgson-Rubinstein-Segerman\,\cite{HRS}. The paper is organized as follows. We recall some basic notions in Section \ref{2}, including hyperideal and flat tetrahedra, angle structures, admissible surfaces and Kojima decompositions. In Section \ref {3} and Section \ref{4}, we respectively prove the two directions of Theorem \ref{main}.\\


\textbf{Acknowledgments:} The first author is supported by a Postgraduate Scholarship Program of China. The authors would like to thank Feng Luo for showing interest and making useful suggestions and Marc Lackenby for bringing our attention to his result in \cite{Lackenby3}. The third author is grateful to Henry Segerman for helpful discussions.

\section{Preliminaries}\label{2}

In this section, we recall some material we need for our results.

\subsection{Hyperideal  tetrahedra}\label{subsection:2.1}

Following Bao-Bonahon\,\cite{BB} and Fuji\,\cite{Fu}, a
\emph{hyperideal tetrahedron} in $\mathbb{H}^3$ is a
compact convex polyhedron that is diffeomorphic to a truncated
tetrahedron in $\mathbb{E}^3$ with four hexagonal faces
right-angled hyperbolic
 hexagons. See Figure \ref{Fig1}. The four hexagonal faces are called the \emph{faces} and the four triangular faces are called the \emph{external faces}. An \emph{edge} in a hyperideal tetrahedron is the
 intersection of two faces, and an \emph{external edge} is the intersection of a
  face and an external face. The \emph{dihedral angle} at an edge is the angle between the two faces adjacent to it. The external faces are isometric to hyperbolic triangles, and the dihedral angle between a face and an external face is $\pi/2.$ By \cite{BB} and \cite{Fu}, a hyperideal tetrahedron in $\mathbb{H}^3$ is determined by its six dihedral angles subject to the constraints that the sum of the dihedral angles at the three edges adjacent to each external face is less than $\pi.$ 

 \begin{figure}[htbp]
\centering
\includegraphics[scale=0.3]{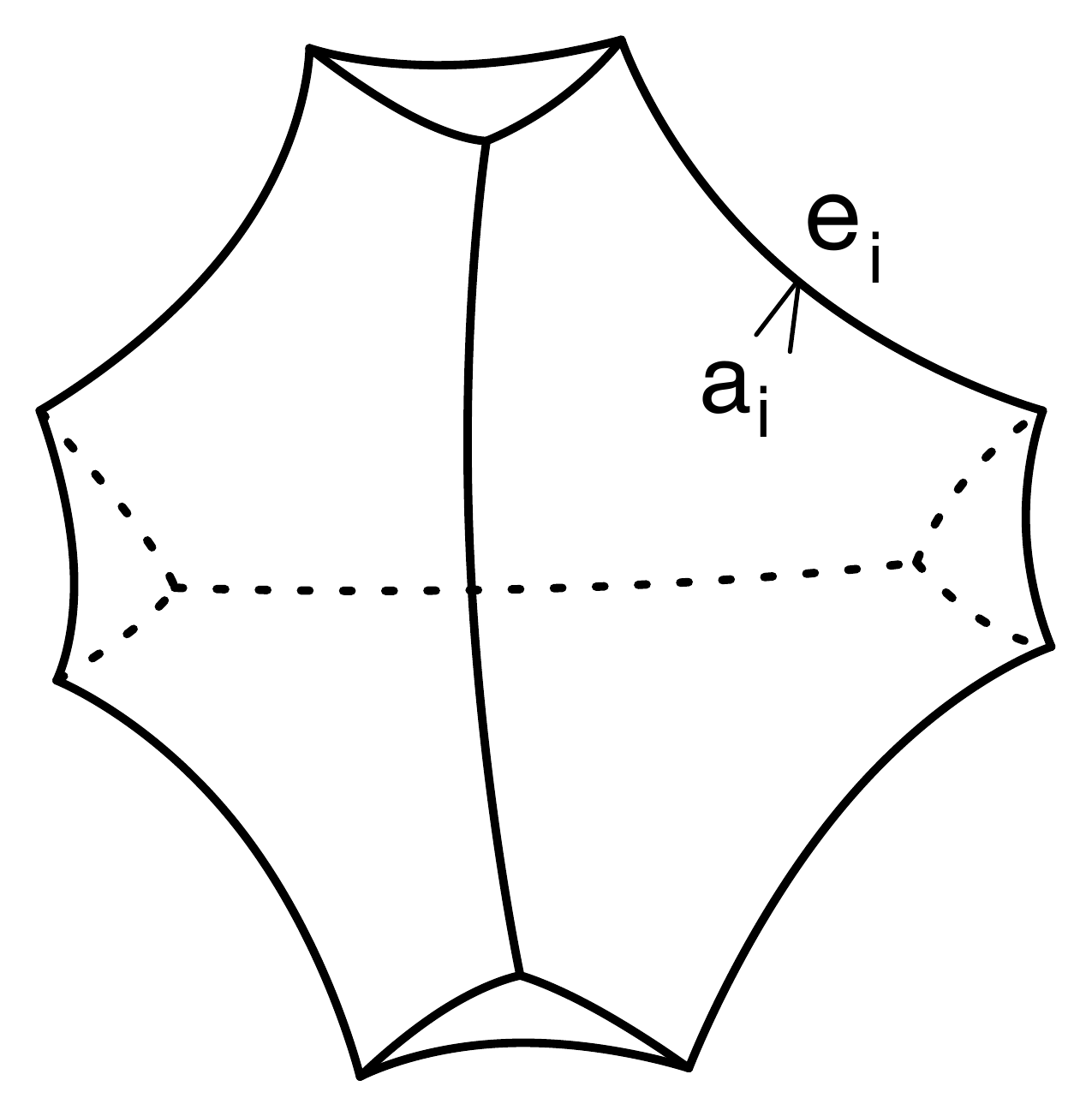}
\caption{Hyperideal tetrahedron}
\label{Fig1}
\end{figure}

\subsection{Flat tetrahedra}

A \emph{flat (hyperideal) tetrahedron} is defined as follows. Let $Q$ be a hyperideal quadrilateral  with edges cyclically labelled as $e_{1},$ $e_{2},$ $e_{3}$ and $e_{4},$ and let $e_{5}$
(resp. $e_{6}$) be the shortest geodesic arc in $Q$ joining
the external edge adjacent to $e_1$ and $e_2$ and the external edge adjacent to $e_3$ and $e_4$ (resp. the shortest geodesic arc in $Q$ joining the external edge adjacent to $e_1$ and $e_4$ and the external edge adjacent to $e_2$ and $e_3$). We call $Q$ with the six
edges $\{e_1,\dots,e_6\}$ a flat tetrahedron. See Figure \ref{Fig4}. The dihedral angles at $e_5$ and $e_6$ are defined to be $\pi$ and are defined to be $0$ at all other edges. 

\begin{figure}[htbp]
\centering
\includegraphics[scale=0.35]{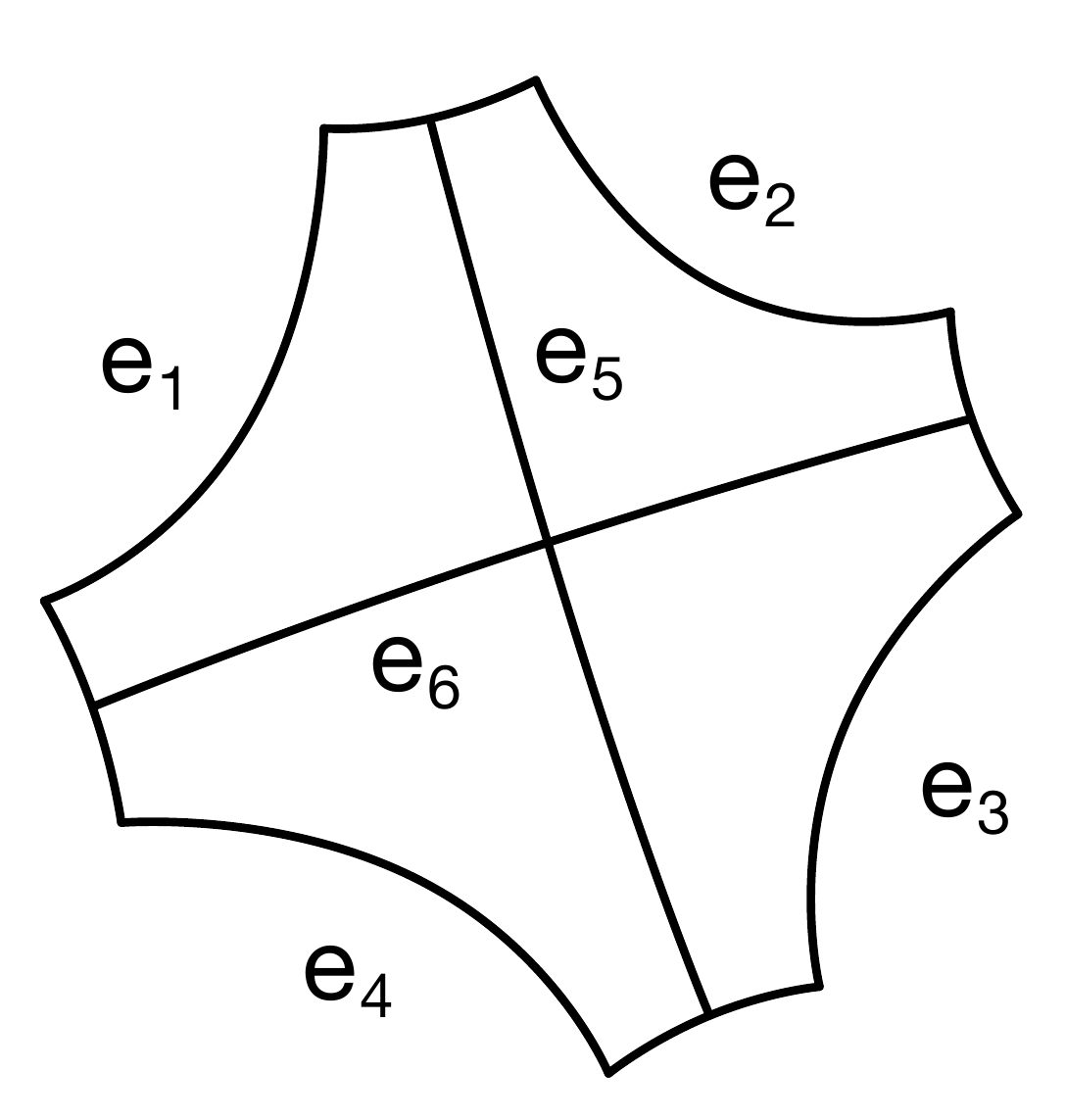}
\caption{Flat tetrahedron}
\label{Fig4}
\end{figure}


\subsection{Ideal triangulations}

Let $M$ be a compact $3$-manifold with boundary consisting of surfaces of negative Euler characteristic. An \emph{ideal triangulation} $\mathcal{T}$ of $M$ consists of a disjoint union $X=\bigsqcup \sigma_i$ of finitely many oriented truncated Euclidean tetrahedra $\sigma_i$ and a collection of orientation reversing affine homeomorphisms $\Phi$ between pairs of hexagonal faces in $X.$ The quotient space $X/\Phi$ is a compact $3$-manifold $M$ with a triangulation $\mathcal T$ where the boundary of $M$ is the quotient of the triangular faces in $X.$ The edges of a truncated Euclidean tetrahedra are the intersection of two hexagonal faces and the external edges are the intersection of a triangular face and a hexagonal face. The \emph{edges} and \emph{external edges} in $\mathcal{T}$ are respectively the quotients of the edges and the external edges in $X,$ the \emph{faces} and the \emph{external faces} in $\mathcal T$ are respectively the quotient of the hexagonal faces and the triangular faces in $X,$ and the \emph{tetrahedra} in $\mathcal{T}$ are the quotients of the truncated Euclidean tetrahedra in $X.$

\subsection{Angle structures}

An \emph{angle structure} on $(M,\mathcal T)$ makes each individual tetrahedron in $\mathcal T$ a hyperideal tetrahedron in $\mathbb{H}^3$ so that the sum of dihedral angles around each edge is $2\pi.$ More precisely, let $E$ and $T$ respectively be the sets of edges and tetrahedra in $\mathcal T$. A pair $(e,\sigma)\in E\times T$ such that $e\subset\sigma$ is called \emph{corner} of $\sigma$ at $e.$ Following Luo\,\cite{Luo1}, an \emph{angle structure} on $(M,\mathcal T)$ is an assignment of a positive real number, called the \emph{dihedral angle}, to each corner so that
\begin{enumerate}[(1)]\label{definition}
\item the sum of the dihedral angles assigned to the corners $(e, \sigma_1),\dots, (e,\sigma_k)$ at each edge $e$ is equal to $2\pi,$ and
\item the sum of the dihedral angles assigned to the corners $(e_1, \sigma),$ $(e_2, \sigma)$ and $(e_3, \sigma)$ at the three edges $e_1,$ $e_2$ and $e_3$ adjacent to each external face of $\sigma$  is less than $\pi.$
\end{enumerate}


\subsection{Admissible surfaces}\label{5}

The proof of that the existence of angle structures implies the existence of hyperbolic structure relies on Lackenby's admissible surface theory for an ideally triangulated $3$-manifold with boundary. Following Lackenby\,\cite{Lackenby1, Lackenby2}, a surface $\Sigma$ in an ideally triangulated $3$-manifold $(M,\mathcal T)$ is \emph{admissible} if for each tetrahedron $\sigma$ in $\mathcal T,$ $\Sigma\cap\sigma$ is a collection of disks embedded in $\sigma,$ called \emph{admissible disks}, such that 

\begin{enumerate}[(1)]
\item  the boundary of each admissible disk is a circle in  the boundary of $\sigma$ that intersects the 1-skeleton of $\partial\sigma,$

\item no arc of intersection between an admissible disk and a face of $\sigma$ has endpoints lying in the same 1-cell of $\partial\sigma$ or in adjacent 1-cells of  $\partial\sigma,$ and

\item no arc of intersection between an admissible disk and an external face of $\sigma$ has endpoints lying in the same 1-cell of $\partial\sigma.$
\end{enumerate}

In \cite{Lackenby1}, Leckenby classified all possible admissible disks that intersect the external faces at most three times, up to isotopy and symmetries of a truncated tetrahedron, as listed in Fig \ref{Fig2}. He also proved that

\begin{figure}[htbp]
\centering
\includegraphics[scale=0.3]{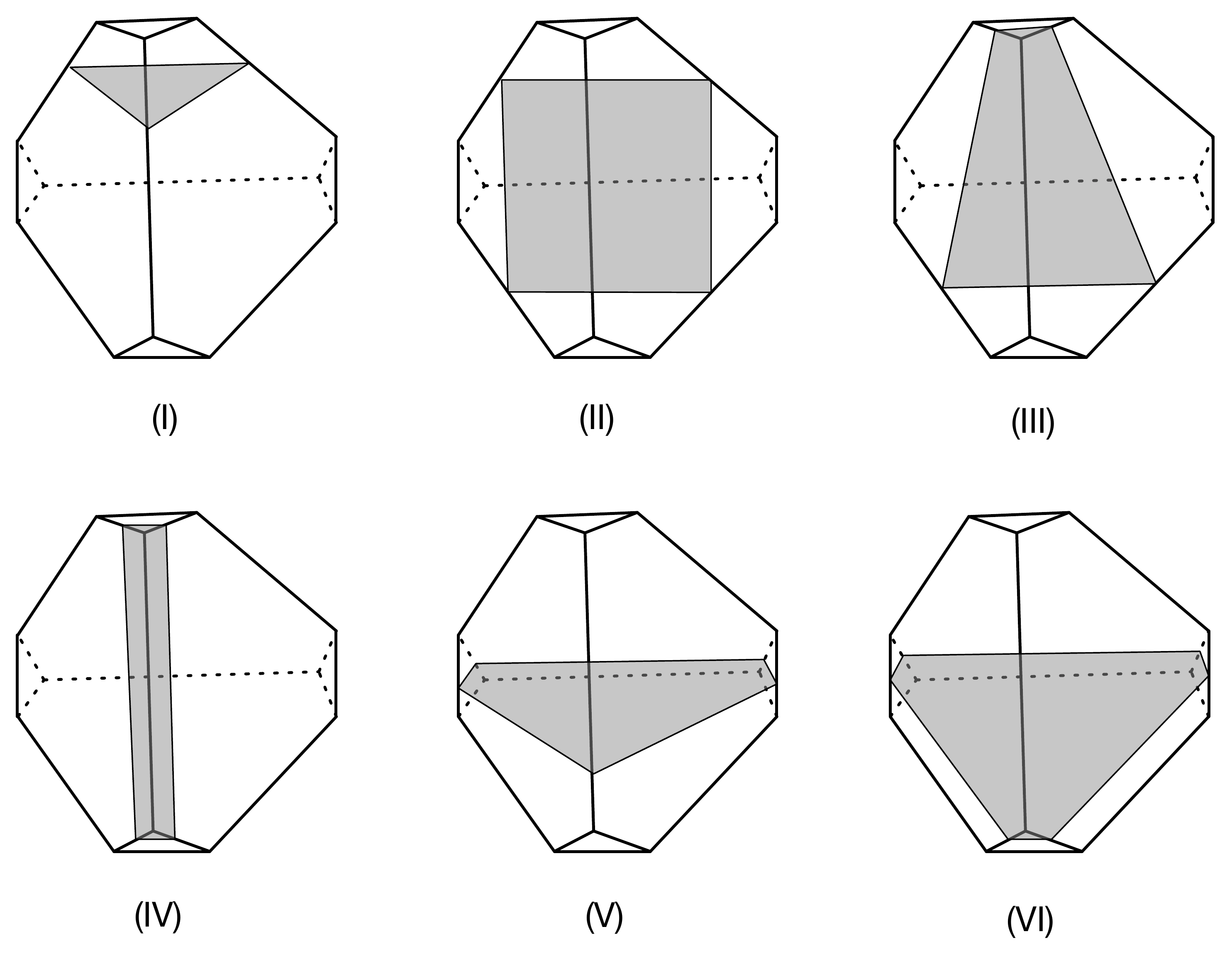}
\caption{Admissible disks}
\label{Fig2}
\end{figure}

\begin{theorem}[Lackenby]\label{thm:Lackenby} If $M$ is a compact irreducible $3$-manifold and $\Sigma$ is a boundary incompressible surface in $M$ that contains no disk component parallel to a disk in $\partial M,$ then $\Sigma$ is homotopic to an immersed admissible surface. 
\end{theorem}


\subsection{Kojima decomposition} 

The proof of the existence of angled triangulations makes a use of Kojima's decomposition of a hyperbolic $3$-manifold with totally geodesic boundary into hyperideal polyhedra. A good way to describe hyperideal polyhedra is to use the projective model $\mathbb H^3\subset\mathbb{RP}^3$ of the hyperbolic $3$-space. In this model, an \emph{untruncated hyperideal polyhedron} is the intersection of $\mathbb H^3$ with a projective polyhedron $Q$ in $\mathbb{RP}^3$ whose vertices are all outside $\mathbb H^3$ and whose edges all meet $\mathbb H^3.$ It is known that for each vertex $v$ of $Q,$ there exists a unique totally geodesic plane $H_v$ that is perpendicular to all the faces of $Q$ adjacent to $v,$ and all the $H_v$ are mutually disjoint. Let $V$ be the set of vertices of $Q.$ The closure $P$ of the unique component of $Q\setminus\bigsqcup_{v\in V} H_v$ that lies in $\mathbb H^3$ is called a \emph{hyperideal polyhedron}, and $Q$ is called the \emph{untruncated polyhedron} of $P.$ From the construction, the faces of $P$ consist of the intersections of $P$ with the faces of $Q$ and the intersection of $Q$ with $H_v.$ We respectively call the former the \emph{faces} and the later the \emph{external faces} of $P.$ The faces of $P$ are all right-angled hyperbolic polygons and the external faces of $P$ are all hyperbolic polygons. Note that a hyperideal tetrahedron is a hyperideal polyhedron with its untruncated polyhedron a tetrahedral in $\mathbb{RP}^3.$  In \cite{Kojima}, Kojima proved that 
 \begin{theorem}[Kojima]
 Every hyperbolic $3$-manifold with totally geodesic boundary is a quotient of finitely many hyperideal polyhedra in $\mathbb H^3$ by isometries between pairs of faces.
 \end{theorem}


\section{Angle structures imply hyperbolic structure }\label{3}

This direction was first know in Lackenby\,\cite{Lackenby3}. We include a proof in this section tor the readers convenience. By Thurston's Hyperbolization Theorem\,\cite{Thurston2}, if a compact $3$-manifold with nonempty boundary is irreducible, atoroidal and contains no essential disks, then it has a hyperbolic structure. Therefore, to prove Theorem \ref{main}, it suffices to prove the following

\begin{theorem}\label{thm:3.1} Let $M$ be a compact $3$-manifold with boundary consisting of surfaces of negative Euler characteristic. If there exists an ideal triangulation of $M$ that admits an angle structure, then $M$ is irreducible, atoroidal and contains no essential disks.
\end{theorem}

Suppose $(M,\mathcal T)$ is an ideally triangulated compact $3$-manifold with boundary consisting of surfaces of negative Euler characteristics and $\Sigma$ is an admissible surface in $(M,\mathcal T).$ Then $\mathcal T$ induces a cell decomposition $\mathcal D$ of $\Sigma$ with $0$-cells the intersections of $\Sigma$ with the edges and the external edges of $\mathcal T,$ $1$-cells the intersections of $\Sigma$ with the faces and the external faces of $\mathcal T$ and $2$-cells the admissible disks of $\Sigma.$ We let $V$ and $F$ respectively be the set of $0$-cells and $2$-cells of $\mathcal D,$ and call a pair $(v, D)\in V\times F$ such that $v\in D$ a \emph{corner} of $D$ at $v.$ We call a $0$-cell of $\mathcal D$ \emph{external} if it is the intersection of $\Sigma$ with an external edge of $\mathcal T,$ and \emph{internal} if otherwise. A corner $(v,D)$ is \emph{external} if $v$ is external, and is \emph{internal} if otherwise. Note that each internal corner $(v,D)$ in $\mathcal D$ corresponds to a unique corner $(e,\sigma)$ in $\mathcal T$ such that $v\in\Sigma\cap e$ and $\sigma\subset\Sigma\cap \sigma.$ An angle structure $\alpha$ on $(M,\mathcal T)$ induces an assignment $\theta$ of positive real numbers, called the \emph{inner angles}, to the corners in $\mathcal D$ by letting $\theta(v,D)=\alpha(e,\sigma)$ if $(v,D)$ is internal, and letting $\theta(v,D)=\pi/2$ if $(v,D)$ is external.

\begin{lemma}\label{3.2}
\begin{enumerate}[(1)]
\item the sum of the inner angles $\theta(v,D_i),$ $i=1,\dots,k,$ assigned to all the corners at an inner $0$-cell $v$ equals $2\pi,$ 
\item the sum of the inner angles $\theta(v_i,D),$ $i=1,\dots, k,$ assigned to all the corners of a admissible disk $D$ with $k$ edges is less than or equal to $(k-2)\pi,$ and
\item  the equality in (2) holds if and only if $D$ is a quadrilateral with four external corners, i.e., (IV) of Figure \ref{Fig2}.

\end{enumerate}
\end{lemma}

\begin{proof}
For (1), we let $e$ be the edge of $\mathcal T$ such that $v\in\Sigma\cap e,$ and let $\sigma_i,$ $i=1,\dots,k,$ be the tetrahedron in $\mathcal T$ such that $D_i\subset\Sigma\cap\sigma_i.$ Then $\sum_{i=1}^k\theta(v,D_i)=\sum_{i=1}^k\alpha(e,\sigma_i)=2\pi.$ For (2), we consider the following two cases. 

Case 1: If $D$ intersect at most one external faces of $\sigma,$ then all the possibilities are listed in ((I), (II) and (III) of) Figure \ref{Fig2}. We respectively call the admissible disks in (I), (II) and (III) of Figure \ref{Fig2} of type I, II and III. Let $\sigma$ be the tetrahedron in $\mathcal T$ such that $D=\Sigma\cap\sigma,$ and let $\alpha(e_i,\sigma),$ $i=1,\dots,6,$ be the dihedral angles assigned by $\alpha$ to the corners $(e_i, \sigma).$ In each case, by renaming the edges, we may assume that $\theta(v_i,D)=\alpha(e_i,\sigma).$ If $D$ is of type I, then $\sum_{i=1}^3\theta(v_i,D)=\sum_{i=1}^3\alpha(e_i,\sigma)<\pi.$ If $D$ is of type II, then $\sum_{i=1}^4\theta(v_i,D)=\sum_{i=1}^4\alpha(e_i,\sigma)<\sum_{i=1}^6\alpha(e_i,\alpha)<2\pi.$ Finally suppose $D$ is of type III with $(v_1,D)$ and $(v_2,D)$ the two internal corners and $(v_3,D)$ and $(v_4,D)$ the two external corners. Then $\theta(v_1,D)+\theta(v_2,D)<\pi$ and $\theta(v_3,D)+\theta(v_4,D)=\pi,$ and hence $\sum_{i=1}^4\theta(v_i,D)<2\pi.$ 

 Case 2: If $D$ intersects at least two external faces of $\sigma,$ then it contains at least four external corners, say $(v_1,D),\dots, (v_4,D),$ whose inner angles add up to $2\pi.$ Thus, $\sum_{i=1}^k\theta(v_i,D)=2\pi+\sum_{i=5}^k\theta(v_i,D)\leqslant 2\pi+(k-4)\pi=(k-2)\pi.$

For (3), we call the admissible disk in (IV) of Figure \ref{Fig2} of type IV. From Case 1 and 2 above, we see that the equality holds if and only if $D$ contains at least four external corners and $k-4=0,$ which is exactly of type IV.\end{proof}

\begin{proposition}\label{thm:3.2} Suppose $(M,\mathcal T)$ is an ideally triangulated $3$-manifold with boundary consisting of surfaces with negative Euler characteristic and $\Sigma$ is an admissible surface in $(M,\mathcal T).$ If $\mathcal T$ admits angle structures, then the Euler characteristic $\chi(\Sigma)\leqslant0,$ and the equality holds if and only if $\Sigma$ is the closure of the intersection of the interior of $M$ with the boundary of a tubular neighborhood  of some edges of $\mathcal T.$ 
\end{proposition}

\begin{proof} Let $\theta$ be the assignment of inner angles to the corners in $\mathcal D$ induced by an angle structure on $(M,\mathcal T),$ and let $V,$ $E,$ $F$ and $F_k,$ $k=1,\dots,n,$ respectively be the set of $0$-cells, $1$-cells, admissible disks and admissible disks with $k$ edges in the cell decomposition $\mathcal D$ of $\Sigma$ induced by $\mathcal T.$ Here $n$ is the maximum number of edges that an admissible disk in $\mathcal D$ has. Then $|F|=\sum_{k=3}^n|F_k|$ and $2|E|=\sum_{k=3}^nk|F_k|.$  If $\Sigma$ is a closed surface, then $\chi(\Sigma)=|V|-\sum_{k=3}^n\frac{k-2}{2}|F_{k}|$. We have
\begin{equation*}
\begin{split}
2\pi\chi (\Sigma)&=2\pi\big(|V|-\sum_{k=3}^n\frac{k-2}{2}|F_{k}|\big)\\
&=\sum_{v\in V}2\pi-\sum_{k=3}^n\sum_{D\in F_k}(k-2)\pi\\
&=\sum_{\{(v,D)|v\in D\}}\theta(v,D)-\sum_{k=3}^n\sum_{D\in F_k}(k-2)\pi\\
&=\sum_{k=3}^n\sum_{D\in F_k}\big(\sum_{v\in D}\theta(v,D)-(k-2)\pi\big)<0,
\end{split}
\end{equation*}
where the last inequality is by Lemma \ref{3.2}.
If $\Sigma$ has nonempty boundary, then we consider the double $(\widetilde{\Sigma},\widetilde{\mathcal D})$ of $(\Sigma,\mathcal D).$ Since the inner angle at an external corner of $\mathcal D$ is defined to be $\pi/2,$ and at each external $0$-cell of $\mathcal D$ there are exactly two conners, the sum of inner angles at each external $0$-cell in ${\mathcal D}$ is $\pi.$ Thus,  the sum of inner angles at each $0$-cell in $\widetilde{\mathcal D}$ is $2\pi.$ By Lemma \ref{3.2} and the same calculation above, $\chi(\widetilde\Sigma)\leqslant 0,$ and the equality holds if and only if all the admissible disks are of type IV, i.e., (IV) in Figure \ref{Fig2}. Therefore, $\chi(\Sigma)=\chi(\widetilde\Sigma)/2\leqslant0,$ and the equality holds if and only if $\Sigma$ is a union of admissible disks of type IV, i.e., the closure of the intersection of the interior of $M$ with the boundary of a tubular neighborhood of some edges in $\mathcal T.$
\end{proof}

\begin{proof}[Proof of Theorem \ref{thm:3.1}] Let $\mathcal T$ be an ideal triangulation of $M$ that admits angle structures. Suppose $M$ contains an essential $2$-sphere, then $M$ contains a normal $2$-sphere $S$ with respect to $\mathcal T.$ By Proposition \ref{3.2}, $\chi(S)<0,$ which is a contradiction. Hence $M$ is irreducible. Now suppose $M$ contains essential tori or disks. Then by Theorem \ref{thm:Lackenby}, $M$ contains admissible tori or disks. By Proposition \ref{3.2}, those admissible surfaces have negative Euler characteristic, which is a contradiction.
\end{proof}


\section{Hyperbolic structure implies angled triangulation}\label{4}

Let $M$ be a hyperbolic 3-manifold with totally geodesic boundary. Our approach of finding angled triangulations consists of the following two steps. In the first step we construct an ideal triangulation $\mathcal T$ of $M$ by subdividing the hyperideal polyhedra in a Kojima decomposition of $M$ into hyperideal tetrahedra, and inserting flat tetrahedra between pairs of faces of the polyhedra. The triangulation inherits non-negative dihedral angles because each tetrahedra in $\mathcal T$ is either hyperideal or flat, so that the sum of dihedral angles around each edge equals $2\pi.$ This construction is combinatorially the same as that in Hodgson-Rubinstein-Segerman\,\cite{HRS} and Lackenby\,\cite{Lackenby3}. In the second step we obtain an angle structure by deforming those non-negative dihedral angles of $\mathcal T$ into positive angles.


\subsection{Existence of partially flat angled triangulations}

In this subsection, we describe Hodgson-Rubinstein-Segerman and Leckenby's algorithm of constructing an ideal triangulation of $M$ from a Kojima decomposition.

We use the projective model $\mathbb H^3\subset \mathbb {RP}^3$ of the hyperbolic $3$-space. As a convention, all the polyhedra and polygons in this subsection will respectively mean untruncated hyperideal polyhedra and untruncated hyperideal polygons.  Recall that a polyhedron is a \emph{pyramid} if its faces consist of an $n$-gon and $n$ triangles which are the cone of the boundary of the $n$-gon to a vertex $v\in\mathbb{RP}^3.$ The vertex $v$ and the $n$-gon are respectively called the \emph{tip} and the \emph{base} of the pyramid.  Let $\mathcal P$ be a Kojima decomposition of $M.$ For each hyperideal polyhedron $P$ in $\mathcal P,$ we let $Q\subset\mathbb {RP}^3$ be the untruncated polyhedron of $P,$ and we arbitrarily pick a vertex $v$ of $Q.$ By taking cone at $v,$ one gets a decomposition of $Q$  into pyramids with tips the vertex $v$ and bases the faces of $Q$ disjoint from $v.$  For the base $D$ of a pyramid, we arbitrarily pick a vertex $w$ of $D$ and decompose $D$ into triangles by taking cone at $w.$ The decomposition of $D$ extends to a decomposition of the pyramid into tetrahedra. In this way, $Q$ is decomposed into a union of tetrahedra. In turns, the intersections of $P$ with those tetrahedra give rise to a decomposition of $P$ into hyperideal tetrahedra. By the construction, each face of $\mathcal P$ is decomposed into hyperideal triangles. In general, the decompositions of a face from two different polyhedra adjacent to it do not always match. In this situation, we insert flat tetrahedra to get a Layered triangulation as follows. Let $v$ and $v'$ respectively be the vertices of a face $D$ where the cone is taken at from the two different polyhedra, and let $u_1,\dots,u_i$ and $w_1,\dots,w_j$ respectively the other vertices of $D$ in the order along the boundary of $D$ from $v$ to $v'.$ See Figure \ref{Fig3}. For each diagonal switch from $u_kv'$ to $u_{k+1}v,$ $k\in\{1,\dots, i-1\},$ and each diagonal switch from $w_kv'$ to $w_{k+1}v,$ $k\in\{1,\dots,j-1\},$ we respectively insert a flat tetrahedron with vertices $\{v, u_i,u_{i+1},v'\}$ and a flat tetrahedron with vertices $\{v,w_k,w_{k+1},v'\}.$

 \begin{figure}[htbp]
\centering
\includegraphics[scale=0.3]{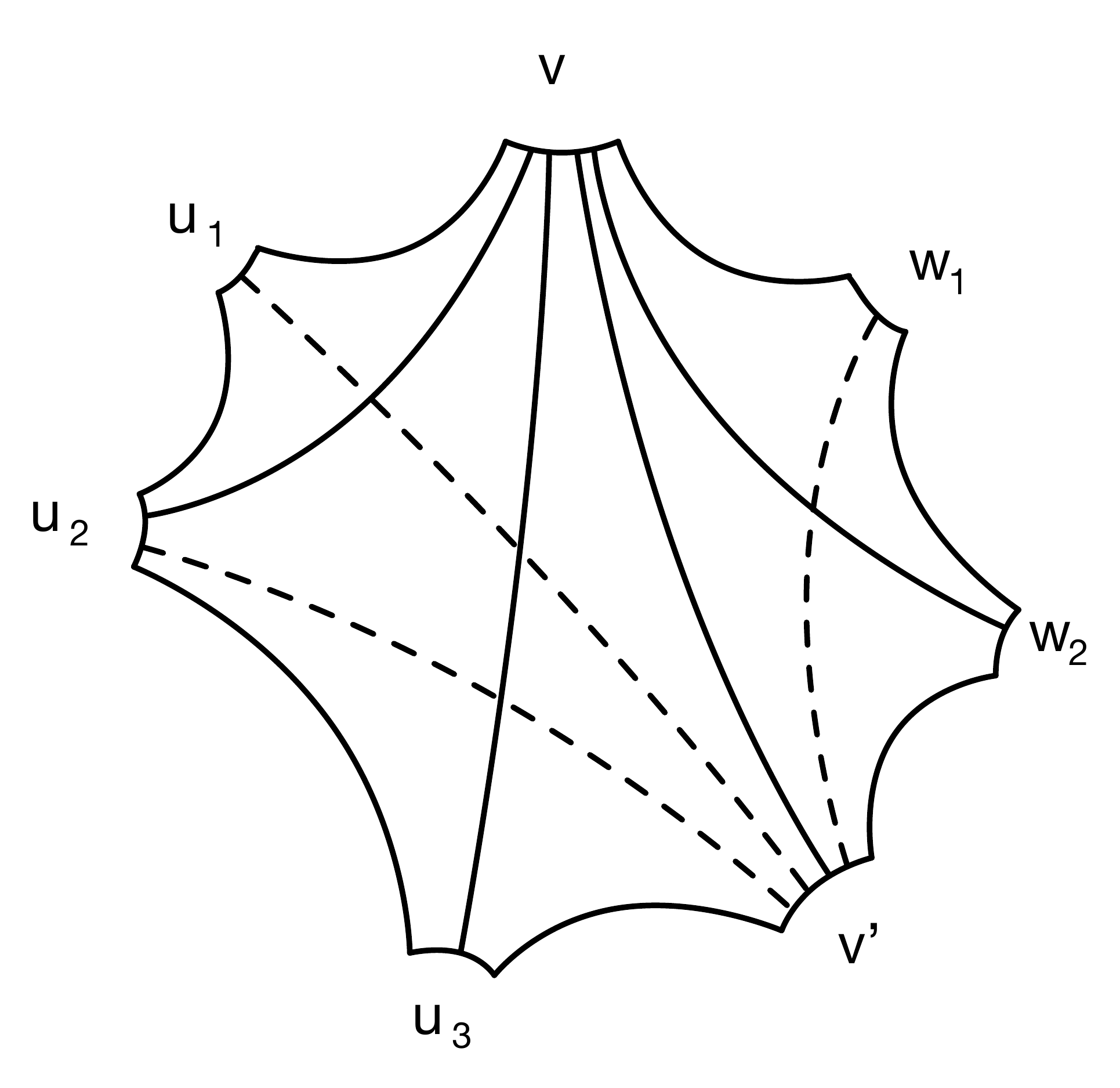}
\caption{Layered triangulation. There are three flat tetrahedra in this Figure respectively with the sets of vertices $\{v,u_1,u_2,v'\},$ $\{v,u_2,u_3,v'\}$ and $\{v,w_1,w_2,v'\}.$
}
\label{Fig3}
\end{figure}

In this way, we get an ideal triangulation $\mathcal T$ of $M$ with tetrahedra the hyperideal tetrahedra obtained by subdividing the hyperideal polyhedra and the flat tetrahedra corresponding to the diagonal switches. From the construction, each corner $(e,\sigma)$ in $\mathcal T$ inherits a non-negative dihedral angle $\beta(e,\sigma)$ so that $\beta(e,\sigma)\in(0,\pi)$ if $\sigma$ is a hyperideal tetrahedron and $\beta(e,\sigma)=0$ or $\pi$ if $\sigma$ is flat. Since no new edges were introduced by doing diagonal switches, each edge in $\mathcal T$ is adjacent to at least one hyperideal tetrahedron. As a consequence, we have

\begin{lemma}\label{k} For each $e\in E,$ there is at least one corner $(e,\sigma)$ such that $\beta(e,\sigma)\in(0,\pi).$
\end{lemma} 

\subsection{Existence of angled triangulations}
In this section, we prove the other direction of Theorem \ref{main} by deforming the non-negative dihedral angles $\beta$ into an angle structure. For each edge $e$ in $\mathcal T,$ we let $m(e),$ $n(e)$ and $k(e)$ respectively be the number of $0$-angles, $\pi$-angles and angles in $(0,\pi)$ around $e.$ Let $t>0.$ We assign a number $\alpha_t(e,\sigma)$ to the corner $(e,\sigma)$ as follows. If $\beta(e,\sigma)=0,$ we let $\alpha_t(e,\sigma)=t,$ if $\beta(e,\sigma)=\pi,$ we let $\alpha_t(e,\sigma)=\pi-3t$ and if $\beta(e,\sigma)\in(0,\pi),$ 
we let $\alpha_t(e,\sigma)=\beta(e,\sigma)-\frac{m(e)-3n(e)}{k(e)}t.$ 

\begin{theorem}\label{4.2} If $t$ is sufficiently small, then $\alpha_t$ defines an angle structure on $(M,\mathcal T).$
\end{theorem}
\begin{proof}
When $t$ is sufficiently small, $\alpha_t(e,\sigma)$ is positive for each corner $(e,\sigma).$ By Lemma \ref{k}, $k(e)>0$ for each edge $e,$ hence $\sum_{\sigma\supset e}\alpha_t(e,\sigma)=m(e)t+n(e)(\pi-3t)+\sum_{\beta(e,\sigma)\in(0,\pi)}\big(\beta(e,\sigma)-\frac{m(e)-3n(e)}{k(e)}t\big)=n(e)\pi+\sum_{\beta(e,\sigma)\in(0,\pi)}\beta(e,\sigma)=2\pi.$ Now suppose $e_1,$ $e_2$ and $e_3$ are the three edges of a tetrahedron $\sigma$ adjacent to an external face. If $\sigma$ is flat, then $\sum_{i=1}^3\alpha_t(e_i,\sigma)=t+t+(\pi-3t)=\pi-t<\pi.$ If $\sigma$ is hyperideal, then $\sum_{i=1}^3\alpha_t(e_i,\sigma)=\sum_{i=1}^3\beta(e_i,\sigma)-\sum_{i=1}^3\frac{m(e_i)-3n(e_i)}{k(e_i)}t.$ Since $\sum_{i=1}^3\beta(e_i,\sigma)<\pi,$ $\sum_{i=1}^3\beta(e_i,\sigma)-\sum_{i=1}^3\frac{m(e_i)-3n(e_i)}{k(e_i)}t<\pi$ when $t$ is sufficiently small. Therefore, when $t$ is sufficiently small, $\alpha_t$ satisfies the conditions of an angle structure.
\end{proof}


\bibliographystyle{amsplain}

\noindent
\noindent
Faze Zhang\\
School of Mathematical Sciences, Dalian University of Technology\\
Dalian, Liaoning, 116024, P.R.China\\
(fazez85@mail.dlut.edu.cn)
\\

\noindent
Ruifeng Qiu\\
Department of Mathematics, East China Normal University\\
Shanghai, 200241, P.R.China\\
(rfqiu@math.ecnu.edu.cn)
\\

\noindent
Tian Yang\\
Department of Mathematics, Stanford University\\
Stanford, CA 94305, USA\\
(yangtian@math.stanford.edu)

\end{document}